\documentclass[12pt]{amsart}
\usepackage{a4wide}
\usepackage[utf8]{inputenc}

 \newtheorem{theorem}{Theorem}[section]
 \newtheorem{prop}[theorem]{Proposition}
 
 \newtheorem{lemma}[theorem]{Lemma}

 \theoremstyle{definition}
 \newtheorem{defi}[theorem]{Definition}
 \newtheorem{ex}[theorem]{Example}
 \newtheorem{rem}[theorem]{Remark}

\begin{document}

\title[Orthogonal coordinates]{Non-existence of orthogonal coordinates on the complex and quaternionic projective spaces}
\author{Paul Gauduchon and Andrei Moroianu}

\address{Paul Gauduchon \\ CMLS\\ {\'E}cole
  Polytechnique \\ CNRS, Universit\'e Paris-Saclay, 91128 Palaiseau \\ France}
\email{paul.gauduchon@polytechnique.edu}

\address{Andrei Moroianu, Laboratoire de Math\'ematiques d'Orsay, Univ.\ Paris-Sud, CNRS,
Universit\'e Paris-Saclay, 91405 Orsay, France}
\email{andrei.moroianu@math.cnrs.fr}


\begin{abstract} DeTurck and Yang have shown that in the neighbourhood of every point of a $3$-dimensional Riemannian manifold, there exists a system of orthogonal coordinates (that is, whith respect to which the metric has diagonal form). We show that this property does not generalize to higher dimensions. In particular, the complex projective spaces  $\mathbb{CP}^m$ and the quaternionic projective spaces $\mathbb{HP}^q$, endowed with their canonical metrics, do not have local systems of orthogonal coordinates for $m,q\ge 2$.
\end{abstract}
\maketitle

\section{Introduction}

A Riemannian manifold is said to {\em admit orthogonal coordinates} if in the neighbourhood of each point there exists a system of coordinates in which the metric has diagonal form, cf. Definition \ref{defi-ortho}. 

Metrics admitting orthogonal coordinates naturally arise in the theory of orthogonal separable dynamical systems, related to the Hamilton-Jacobi equation, and have been considered by many authors starting with Paul Stäckel \cite{st} and Luther Pfahler Eisenhart \cite{eis}, Charles Boyer \cite{boyer}, Paul Tod \cite{tod}, and more recently, James D. E. Grant and James A. Vickers \cite{gv}, Sergio Benenti \cite{be1}, \cite{be2},  Konrad Schöbel \cite{sch}, and others.

Flat, or, more generally, locally conformally flat Riemannian manifolds (in particular every Riemannian surface) clearly admit orthogonal coordinates. In a beautiful paper published in 1984, Dennis M. DeTurck and Deane Yang \cite{deT-Y} showed that every Riemannian metric of dimension 3 has orthogonal coordinates. In the same paper, they also observe that the existence issue of orthgonal coordinates on Riemannian manifolds of dimension greater than $3$ becomes an overdetermined problem, and therefore one can hardly expect orthogonal coordinates on a generic Riemannian manifold.  On the other hand, the existence/non-existence issue of orthogonal coordinates on a given family of Riemanian manifolds has remained a quite interesting question, albeit largely unexplored.

The aim of this paper is to establish the non-existence of orthogonal coordinates on two classical families of Riemannian manifolds, namely the standard complex projective spaces $\mathbb{CP}^m$ for $m\ge 2$ and the standard quaternionic projective spaces $\mathbb{HP}^q$ for $q\ge 2$. The overall argument relies on some remarkable feature --- already noticed by DeTurck and Yang --- of the curvature of Riemannian manifolds admitting orthogonal coordinates, together with some additional specific arguments in dimension 4, for the complex projective plane $\mathbb{CP}^2$.

A list of open questions is proposed at the end of the paper.

{\sc Acknowledgment.} We are grateful to François Golse for having brought this question to our attention and to Charles Boyer and Paul Tod for communicating to us their previous works and further related references. We warmly thank David Johnson for having pointed out an error in the expression (\ref{R-ijkl}) of the curvature in a previous version of this paper.

\section{Riemannian metrics with orthogonal coordinates} \label{sortho}

Let $(M, g)$ be any Riemannian manifold of dimension $n$. Let $x_1, \ldots, x _n$ be any system of local coordinates defined on some open set $\mathcal{U}$ and denote by $\frac{\partial}{\partial x _1}, \ldots,  \frac{\partial}{\partial x _n}$ the corresponding frame on $\mathcal{U}$; the restriction  to $\mathcal{U}$ of the metric $g$ is then of the form:
\begin{equation} g = \sum _{i, j = 1} ^n g _{i j} d x _i \otimes d x _j, \end{equation}
by setting $g _{ij} := g (\frac{\partial}{\partial x _i}, \frac{\partial}{\partial x _j})$. 

\begin{defi} \label{defi-ortho} The system of coordinates $x_1, \ldots, x _n$ is called {\em orthogonal}, if $g _{i j} = 0$ whenever $i \neq j$, hence if $g$ is of the form
\begin{equation} \label{ortho} g = \sum _{j = 1} ^n a _j ^2 \, d x _j \otimes d x _j, \end{equation}
for some positive functions $a _1, \ldots, a _n$.
We say that a Riemannian manifold $(M,g)$ {\it has orthogonal coordinates}  if every point of $M$ has a neighbourhood on which there exists a system of orthogonal coordinates.
\end{defi}
\begin{rem} \label{rem-b} If a system of orthogonal coordinates $x_1, \ldots, x _n$ exists, any system of coordinates $y_1, \ldots, y _n$ of the form $y _i = \varphi _i (x_i)$, where $\varphi _i$ is a real function whose derivative $\varphi '_i$ has no zero,  is orthogonal as well, since
      \begin{equation} g = \sum _{i = 1}^n b _i ^2 \, d y _i \otimes d y _i, \end{equation}
    with
    \begin{equation} \label{b} b _i = \frac{a _i}{|\varphi '_i (x _i)|}, \end{equation}
    for $i = 1, \ldots, n$.
    \end{rem}
    
\begin{ex} The standard flat metric $g _0$ on $M = \mathbb{R} ^n$ is of the form
    \begin{equation} \label{flat} g _0 = \sum _{i = 1} ^n d x _i \otimes d x _i, \end{equation}
    where the $x _i$'s denote the natural coordinates of $\mathbb{R} ^n$. Conversely, a Riemannian metric $g$ is flat whenever, in the neighbourhood of any point,  there exists a system of coordinates such that $g$ is of this form. \end{ex}

\begin{ex} \label{ex-sphere}  Denote by $\mathbb{S} ^n$ the $n$-dimensional standard unit sphere $$\mathbb{S} ^n = \{u = (u_0, \ldots, u _n) \, | \, \sum _{i = 0} ^n u _i ^2 = 1\},$$ and by $g _S$ the standard Riemannian metric of sectional curvature $1$, induced by the standard flat metric of $\mathbb{R} ^{n + 1}$. Denote by $N$ the point $(1, 0, \ldots, 0)$ of $\mathbb{S} ^n$ and by $\mathcal{U}$ the open set $\mathbb{S} ^n \setminus \{N\}$. Then, on $\mathcal{U}$, the metric $g _S$ is of the form:
    \begin{equation} g _S = \frac{4 \sum _{j = 1} ^n d x _j \otimes d x _j}{(1 + \sum _{j = 1} ^n x _j ^2) ^2}, \end{equation}
    by setting
    \begin{equation} x _j = \frac{u_j}{1 - u _0}, \qquad j = 1, \ldots, n. \end{equation}
      Conversely, any locally conformally flat metric, in particular, any Riemannian metric in dimension $2$,  can be locally written on the form
 \begin{equation} g = a ^2 \, \sum _{j = 1} ^n d x _j \otimes d x _j, \end{equation}
   i.e. on the form (\ref{ortho}), with $a _j = a$, $j = 1, \ldots, n$.
  \end{ex}

  \smallskip

Assume from now on that $(M, g)$ is a Riemannian manifold of dimension $n$, with $n \geq 4$. We assume that $x _1, \ldots, x_n$ is an orthogonal system of coordinates, as defined above, and we denote by $\{e_1, \ldots, e _n\}$ the {\em associated orthonormal frame}, with
\begin{equation} \label{e} e _j := a _j^{-1}\, \frac{\partial}{\partial x_j}, \qquad j = 1, \ldots, n. \end{equation}
Notice that this frame remains unchanged if the orthogonal system $x_1, \ldots, x _n$ is replaced by $y_1, \ldots, y _n$ as in Remark \ref{rem-b}. We denote by $\nabla$ the Levi-Civita connection of $g$ and by $R$ its curvature, defined by
\begin{equation} \label{R} R _{X, Y} Z = \nabla _{[X, Y]} Z - \nabla _X (\nabla _Y Z) + \nabla _Y (\nabla _X Z), \end{equation}
for any vector fields $X, Y, Z$ on $M$. 
\begin{prop} \label{prop-crux} Let $(M, g)$ be a Riemannian manifold of dimension $n \geq 4$, equipped with orthogonal coordinates on some open set $\mathcal{U}$, where the metric is of the form {\rm (\ref{ortho})}. Denote by $\{e_1, \ldots, e _n\}$ the associated orthonormal frame as defined above. Then,
 \begin{equation} \label{nabla-e} \nabla _{e _i} e _j = a _i ^{-1} d a _i (e _j) \, e _i - \delta_{i j} \, a _j ^{-1} (d a _j) ^{\sharp}, \quad i, j = 1, \ldots, n, \end{equation}
where $(d a _j) ^{\sharp}$ denotes the vector field dual to $d a _j$ with respect to $g$ and $\delta _{i j}$ the usual Kronecker symbol. Moreover,
\begin{equation} \label{R-ijkl} \begin{split} g (R _{e_i, e_j} e_k, e_{\ell}) & = \delta _{i \ell} \, a _i ^{-1} (\nabla _{e _j} d a _i) (e _k) - \delta _{j \ell} \, a _j ^{-1} (\nabla _{e _i} d a _j) (e _k) \\ & - \delta _{i k} \, a _k ^{-1} (\nabla _{e _j} d a _k) (e _{\ell}) + \delta _{j k} \, a _k ^{-1} (\nabla _{e _i} d a _k) (e _{\ell})\\ & + (\delta _{i k} \delta _{j \ell} - \delta _{j k} \delta _{i \ell})  \, a _i ^{-1} a _j ^{-1} \, g (d a _i, d a _j),  \end{split} \end{equation}
for any quadruple $i, j, k, \ell = 1, \ldots, n$.
In particular,  for any triple $i, j, k$ with $i \neq j \neq k\neq i$, we have:
\begin{equation} \label{R-ijk} R _{e _i, e _j} e _k = a _i ^{-1} (\nabla _{e_j} d a _i) (e_k) \, e _i - a _j ^{-1} (\nabla _{e_i} d a _j) (e_k) \, e _j, \end{equation}
and, as observed in \cite{deT-Y},  for quadruple $i, j, k,  \ell$ with $i,j,k,\ell$ mutually distinct:
\begin{equation} \label{R-2-ijkl} g (R _{e_i, e _j} e _k, e _{\ell}) = 0. \end{equation}
\end{prop}
\begin{proof} For any $i, j$, we have $[e _i, e _j] = [a _i ^{-1} \frac{\partial}{\partial x _i}, a _j ^{-1} \frac{\partial}{\partial x _j}]$, hence
  \begin{equation} \label{bracket} [e _i, e _j] = a _i ^{-1} d a _i (e_j) \, e_i - a _j ^{-1} d a _j (e_i) \, e _j,  \end{equation}
  whereas the usual {\it Koszul formula} for the Levi-Civita connection is here reduced to
  \begin{equation} \label{koszul} 2 g (\nabla _{e_i} e _j, e _k) = g ([e_i, e_j], e_k) + g  ([e_k, e_i], e_j) + g  (e_i, [e_k, e_j]). \end{equation}
  We easily infer:
 \begin{equation} \label{ij} \begin{split} & \nabla _{e_i} e _j =  a _i ^{-1} d a _i (e _j) \, e _i, \qquad i \neq j, \\ & \nabla _{e _j} e _j  = - \sum _{i \neq j} a _j ^{-1} d a _j (e_i) \, e_i, \end{split} \end{equation}
  hence (\ref{nabla-e}). A straightforward computation then gives (\ref{R-ijkl}), and (\ref{R-ijk})--(\ref{R-2-ijkl})  follow readily.
\end{proof}
\begin{rem}\label{r2}
Equation \eqref{nabla-e} can equivalently be written as
 \begin{equation} \label{alp} \nabla _{e _i} e _j^\flat = e_j\lrcorner (\alpha_i\wedge e_i^\flat), \quad i, j = 1, \ldots, n, \end{equation}
where $\alpha_i:=a_i^{-1}da_i$. Conversely, a (local) orthonormal frame satisfying \eqref{alp} for some $1$-forms $\alpha_i$ is necessarily induced by a system of orthogonal coordinates. Indeed, using \eqref{alp} we can write
$$de_j^\flat=\sum_{i=1}^n e_i^\flat\wedge\nabla _{e _i} e _j^\flat=\sum_{i=1}^n e_i^\flat\wedge(\alpha_i( e _j)e_i^\flat-\delta_{ij}\alpha_i)=\alpha_i\wedge e_i^\flat,$$
whence $e_j^\flat\wedge de_j^\flat=0$ for every $j= 1, \ldots, n$. The Frobenius theorem shows that there exist functions $x_i$ and $b_i$  (defined on some smaller neighbourhood) such that $e_j^\flat=b_jdx_j$ for every $j= 1, \ldots, n$. Changing the sign of $x_j$ if necessary, one can assume that each $b_j$ is a positive function. Then $x_1,\ldots,x_n$ is a system of orthogonal coordinates with associated orthonomal frame $e_1,\ldots,e_n$.
\end{rem}

\section{The complex projective spaces} \label{sprojective} 

We now consider the complex projective space $M = \mathbb{CP} ^m$, $m \geq 2$, equipped with the Fubini-Study metric, $g_{FS}$, of constant holomorphic sectional curvature $c$, whose curvature, $R$, is given by:
\begin{equation} \label{FS} R^{FS} _{X, Y} Z  =  \frac{c}{4} \big(g_{FS} (X, Z) \, Y - g_{FS}(Y, Z) \, X + \omega (X, Z) \, JY - \omega (Y, Z) \, JX  + 2 \omega (X, Y) \, JZ\big)  \end{equation}
for any vector fields $X, Y, Z$, where $J$ denotes the complex structure of $\mathbb{CP} ^m$ and $\omega = g _{FS} (J \cdot, \cdot)$ its K\"ahler form.  For convenience and without loss of generality, we assume that $c = 4$. Our aim  is the show that $\mathbb{CP} ^m$, equipped with the Fubini-Study metric, admits no orthogonal system of coordinates.  Since the case when $m = 2$ requires a specific argument, see below Proposition \ref{prop-2}, we first show:
\begin{prop} \label{prop-m} For $m \geq 3$, the complex projective space $\mathbb{CP} ^m$, equipped with the standard Fubini-Study metric, admits no orthogonal system of coordinates. \end{prop}
\begin{proof} 
  Suppose, for a contradiction, that $\mathbb{CP} ^m$ admits local orthogonal coordinates, i.e. that $g _{FS}$ is of the form (\ref{ortho}) for some local coordinates $x_1, \ldots, x_n$, $n = 2 m$,  on some open set $\mathcal{U}$, and consider the corresponding orthonormal frame $\{e_1, \ldots, e_n\}$ as in Proposition  \ref{prop-crux}. 
  
  Choose any pair $e_i, e_j$ such that $\omega (e_i, e _j) \neq 0$, and any $e _k$ orthogonal to $e_i$ and $e _j$.  In view of (\ref{FS}), with $c = 4$,  we have
  \begin{equation} R ^{FS} _{e_i, e_j} e _k = \omega (e_i, e_k) \, J e_j - \omega (e_j, e_k) \, Je_i + 2 \omega (e_i, e_j) \, J e_k, \end{equation}
  whereas, by (\ref{R-ijk}), we should have:
  \begin{equation} R ^{FS} _{e_i, e _j} e _k = f  _i  \, e_i - f  _j  \, e_j, \end{equation}
with $f  _i := a _i ^{-1} (\nabla _{e_j} d a _i) (e_k)$, $f  _j := a _j ^{-1} (\nabla _{e_i} d a _j) (e_k)$, so that:
\begin{equation} 2 \omega (e_i, e_j) \, e _k = - \omega (e_i, e_k) \, e_j + \omega (e_j, e_k) \, e_i - f  _i \, J e _i + f  _j \, J e_j. \end{equation}
Since $e_k$ is orthogonal to $e_i, e_j$, the functions $f  _i, f  _j$ are necessarily given by  $f  _i = - \frac{\omega (e_i, e_k)}{\omega (e_i, e_j)}$ and $f  _j = \frac{\omega (e_j, e_k)}{\omega (e_i, e_j)}$, whence
\begin{equation*} e _k = \frac{\omega (e _i, e _k)}{2 (\omega (e_i, e _j))^2} \, (- \omega (e_i, e _j) \, e_j + J e _i) + \frac{\omega (e _j, e _k)}{2 (\omega (e_i, e _j))^2} \, ( \omega (e_i, e _j) \, e_i + J e _j). \end{equation*}
Since $e _k$ may be any element in  the orthonormal frame $e_1, \ldots, e _n$ distinct from $e_i, e_j$, this means that the $(2 m - 2)$-dimensional space orthogonal  to the $2$-dimensional space generated by $e_i, e_j$ would be contained in the $2$-dimensional space generated by $- \omega (e_i, e _j) \, e_j + J e _i$ and  $\omega (e_i, e _j) \, e_i + J e _j$. This clearly cannot hold unless $m = 2$.
\end{proof}
We now show:
\begin{prop} \label{prop-2} The complex projective plane $\mathbb{CP} ^2$, equipped with the standard Fubini-Study metric, admits no orthogonal system of coordinates. \end{prop}
\begin{proof} Again, assume for a contradiction, that $\mathbb{CP} ^2$ admits local orthogonal coordinates $x _1, x_2, x_3, x_4$ on some open set $\mathcal{U}$ and denote by $e_1, e_2, e_4, e_4$ the corresponding orthonormal frame. As for any direct orthonormal frame  relative to  the orientation induced by the natural complex structure $J$ of $\mathbb{CP} ^2$, we have
\begin{equation} \omega (e_1, e_2) \omega (e_3, e_4) - \omega (e_1, e_3) \omega (e_2, e_4) + \omega (e_1, e_4) \omega (e_2, e_3) = 1, \end{equation}
since the volume form of $g _{FS}$ for the chosen orientation is $\frac{\omega \wedge \omega}{2}$. By (\ref{FS}), with $c = 4$, it follows that 
\begin{equation} \label{RFS} \begin{split} g_{FS} (R^{FS} _{e_1, e_2} e_3, e_4) & =
    \omega (e_1, e_3) \omega (e_2, e_4) - \omega (e_1, e _4) \omega (e_2, e _3)\\ &  + 2 \omega (e_1, e_2) \omega (e_3, e_4) \\ & = - 1 + 3 \omega (e_1, e_2) \omega (e_3, e _4) \\ & = - 1 + 3 \big(\omega (e _1, e_2)\big) ^2, \end{split} \end{equation}
as $\omega$ is self-dual.
In view of (\ref{R-2-ijkl}) in Proposition \ref{prop-crux} and of (\ref{RFS}), we have
\begin{equation} \label{omega-13} \big(\omega (e _i, e _j)\big)^2 = \frac{1}{3}, \end{equation}
for any $i, j$, $i \neq j$. Up to possibly  changing $J$ into $-J$, we may then arrange that
\begin{equation} \label{omega-e} \begin{split}  \omega & (e_1, e_2)  = \omega (e_3, e_4) = \omega (e_1, e_3)\\ &  = - \omega (e_2, e_4) = \omega (e_1, e_4) = \omega(e_2, e_3)
    = \frac{1}{\sqrt{3}}, \end{split} \end{equation}
i.e. that
\begin{equation} \label{J} \begin{split} & J e_1 = \frac{e_2 + e_3 + e_4}{\sqrt{3}}, \quad J e_2 = \frac{- e_1 + e_3 - e_4}{\sqrt{3}}, \\ & J e_3 =
    \frac{- e_1 - e_2 + e _4}{\sqrt{3}}, \quad J e _4 = \frac{- e_1 + e_2 - e _3}{\sqrt{3}}. \end{split} \end{equation}
By making explicit the identities $\nabla _{e _i} J e _1 = J \nabla _{e _i} e_1$, $i = 1, 2, 3, 4$, via (\ref{nabla-e}) and (\ref{J}) we easily get:
\begin{equation} \label{a1-bis} d a _1 (e_2) = d a _1 (e_3) = d a _1 (e_4), \end{equation}
\begin{equation} \label{a2-bis} d a _2 (e_1) = - d a _2 (e_3) = d a _2 (e_4), \end{equation}
\begin{equation} \label{a3-bis} d a _3 (e_1) = d a _3 (e_2) = - d a _3 (e_4), \end{equation}
\begin{equation} \label{a4-bis} d a _4 (e_1) = - d a _4 (e_2) = d a _4 (e_3).  \end{equation}
From  (\ref{a1-bis}) we infer that the vector fields  $e_2 - e_3$, and $e _2 - e_4$ both belong to the kernel of $d a _1$; it follows that their bracket $- [e_2, e_4] + [e_2, e_3] + [e_3, e_4]$, which, by (\ref{nabla-e}) is equal to $2 a _2 ^{-1} d a _2 (e_3) \, e_2 + 2 a _3 ^{-1} d a _3 (e_4) \, e _3 + 2 a _4 ^{-1} d a _4 (e_2) \, e _4$, also belongs to the kernel of $d a _1$, so that: $a _2 ^{-1} d a _2 (e_3) d a _1 (e_2) + a _3 ^{-1} d a _3 (e_4) d a _1 (e _3) +  a _4 ^{-1} d a _4 (e_2) d a _1 (e _4) = 0$.
By introducing the notation
\begin{equation} \label{notation} \begin{split} & c _1 := a _1 ^{-1} d a _1 (e _2) = a _1 ^{-1} d a _1 (e _3) = a _1 ^{-1} d a _1 (e _4), \\ &
    c _2 := a _2 ^{-1} d a _2 (e _1) =  - a _2 ^{-1} d a _2 (e _3) = a _2 ^{-1} d a _2 (e _4), \\ & c _3 := a _3 ^{-1} d a _3 (e _1) =   a _3 ^{-1} d a _3 (e _2) =   - a _3 ^{-1} d a _3 (e _4), \\ & c _4 := a _4 ^{-1} d a _4 (e _1) = -  a _4 ^{-1} d a _4 (e _2) =   a _4 ^{-1} d a _4 (e _3), \end{split} \end{equation}
and by using (\ref{a1-bis}) again, this can be rewritten as $(c_2 + c _3 + c _4) \, c _1 = 0$. We thus get the following alternative:
\begin{equation} \label{alt-1} \text{either} \quad c _2 + c _3 + c _4 = 0 \quad \text{or} \quad c _1 = 0. \end{equation}
By considering (\ref{a2-bis}), (\ref{a3-bis}) and (\ref{a4-bis}), we similarly  obtain the following three alternatives:
\begin{equation} \label{alt-2} \text{either} \quad c _1 + c _3 - c _4 = 0 \quad \text{or} \quad c _2 = 0, \end{equation}
\begin{equation} \label{alt-3} \text{either} \quad  c _1  - c _2 + c _4 = 0 \quad \text{or} \quad c _3 = 0, \end{equation}
\begin{equation} \label{alt-4} \text{either} \quad c _1 + c _2 - c _3 = 0 \quad \text{or} \quad c _4 = 0.  \end{equation}
Since the matrix  $\begin{pmatrix} 0 & 1 & 1 & 1 \\ 1 & 0 & 1 & - 1 \\ 1 & - 1 & 0 & 1\\ 1 & 1 & - 1 & 0 \end{pmatrix}$ is invertible, the  left hand sides of (\ref{alt-1}), (\ref{alt-2}), (\ref{alt-3}), (\ref{alt-4}) cannot be all equal to zero, unless all $c _i$ are zero, which would imply that each $a_j$ is  a  function of $x_j$ only, hence that  the Fubin-Study metric $g _{FS}$ is flat. It then follows that $c _i = 0$, for some $i$. As just observed, this implies that $a_i$ is a function of $x_i$ only, and we can then consider that $a _i$ is constant. By (\ref{R-ijk}), this implies that $R ^{FS} _{e _i, e _j} e _k = - a _j ^{-1} (\nabla _{e _i} d a _j) (e _k) e _j$ for any $j \neq  k$, both  distinct from $i$; in particular, we then have:
\begin{equation} \label{F1} g _{FS} (R ^{FS}_{e_i, e_j} e _k, e _i) = 0. \end{equation}
On the other hand, by (\ref{FS}), with $c = 4$, we have that
\begin{equation} R ^{FS} _{e_i, e_j} e _k = \omega(e_i, e_k) \, J e_j - \omega (e_j, e_k) \, Je_i + 2 \omega (e_i, e_j) \, J e _k, \end{equation}
hence $g _{FS} (R ^{FS}_{e_i, e_j} e _k, e _i) = - 3 \omega (e_i, e_j) \, \omega (e_i, e _k)$; from  (\ref{omega-13}), we then infer:
\begin{equation} \label{F2} g _{FS} (R ^{FS}_{e_i, e_j} e _k, e _i) = \pm 1, \end{equation}
which evidently contradicts (\ref{F1}). 
\end{proof}

\section{The quaternionic projective space}

In this section, we consider the quaternionic projective space $\mathbb{H P} ^q$, $q \geq 2$, equipped with its standard quaternionic K\"ahler structure, determined by the Riemannian metric $g$ and a rank $3$ subbundle, $Q$,  of the bundle of skew-symmetric endomorphisms of $T \mathbb{H P} ^q$, preserved by the Levi-Civita of $g$ and locally generated by triplets of almost complex structures, $J _1, J _2, J_3$, such that $J_1 J_2 J_3 = - {\rm Id}$. For any such triplet, we set $\omega _{\alpha} := g (J _{\alpha} \cdot, \cdot)$, $\alpha = 1, 2, 3$. 

If $q= 1$, $\mathbb{HP} ^1$ is isometric, up to scaling, to the standard round sphere $\mathbb{S} ^4$ and therefore  does admit orthogonal coordinates, cf. Example \ref{ex-sphere}. We have however:
\begin{prop} \label{prop-HPk} For $q \geq 2$, the quaternionic projective space $\mathbb{HP} ^q$ admits no local orthogonal coordinates.
\end{prop}
\begin{proof} Up to scaling, the curvature, ${\rm R}$, of $\mathbb{HP} ^q$, viewed as a symmetric  endomorphism of $\Lambda ^2 T \mathbb{HP} ^q$,  is locally given by:
  \begin{equation} \label{R-H} {\rm R} (X \wedge Y) = X \wedge Y + \sum _{\alpha = 1} ^3 J _{\alpha} X \wedge J _{\alpha} Y + 2 \sum _{\alpha = 1} ^3 \omega _{\alpha} (X, Y) \, \omega _{\alpha} ^{\sharp _g}, \end{equation}
  for any vector fields $X, Y$, where $\omega _{\alpha} ^{\sharp _g}$ denotes the section of $\Lambda ^2 T \mathbb{HP} ^q$ determined by  $\omega _{\alpha}$ by Riemannian duality.
  
  Assume for a contradiction, that $\mathbb{HP} ^q$ admits  an orthogonal system of coordinates, $\{x_1, \ldots, x _{4q}\}$,  on some connected open set $\mathcal{U}$ where $Q$ is trivialized by a triplet $J_1, J_2, J_3$ as above, where ${\rm R}$ is then given by (\ref{R-H}), and denote by $\{e_1, \ldots, e _{4 q}\}$ the corresponding orthonormal frame, as defined by (\ref{e}). For convenience, we introduce the notation:
  \begin{equation} \label{a} a _{i j k \ell} := \sum _{\alpha = 1} ^3 \omega _{\alpha} (e_i, e_j) \omega _{\alpha} (e_k, e _{\ell}). \end{equation}
  From (\ref{R-2-ijkl}) and  (\ref{R-H}), we should have
  \begin{equation} 0 = g ({\rm R} (e _i \wedge e _j), e _k \wedge e _{\ell}) = a _{i k j \ell} + a _{k j i \ell} + 2 a _{i j k \ell}, \end{equation}
  for any pairwise distinct $4$-uplets $i, j, k,  \ell$. For any such $4$-uplet, we then infer $ a _{i k j \ell} + a _{k j i \ell} + a _{j i k \ell} = 3 a _{j i k \ell}$.
  Since the left hand side of this identity is invariant by circular permutation of $i, k, j$, we thus obtain:
  \begin{equation}  a _{i k j \ell} = a _{k j i \ell} = a _{j i k \ell}, \end{equation}
  for any pairwise distinct $4$-uplets $i, j, k, \ell$. From the first
  equality in (\ref{a-compare}), we infer that $\sum _{\alpha = 1} ^3 \omega _{\alpha} (e_i, e_k) J _{\alpha} e_j + \sum _{\alpha = 1} ^3 \omega _{\alpha} (e_j, e_k) J _{\alpha} e_i$ is orthogonal to $\ell$, for any $\ell$ distinct from $i, j, k$, so that 
  \begin{equation}\label{a-compare} \sum _{\alpha = 1} ^3 \omega _{\alpha} (e_i, e_k) J _{\alpha} e_j + \sum _{\alpha = 1} ^3 \omega _{\alpha} (e_j, e_k) J _{\alpha} e_i \in {\rm span} (e_i, e _j, e _k), \end{equation}
  for any pairwise distinct triplets $i, j, k$. 
  
  We now fix $i, k$ such that $\omega _1 (e_i, e _k) \neq 0$ (for any fixed $i$, we can obviously chose such a $k$). Denote $b _{\alpha} := \omega _{\alpha} (e_i, e_k)$. Then $b _1 \neq 0$ and the endomorphism
  $${J} := \frac{b_1J_1+b_2J_2+b_3J_3}{\sqrt{b_1^2+b_2^2+b_3^2}}$$ is a well-defined section of $Q$ on $\mathcal{U}$. From (\ref{a-compare}), we get:
  \begin{equation} \label{barJ-1} {J} e _j \in {\rm span} (e_i, e_j, e _k, J_1 e_i, J_2 e_i, J _3 e _i), \end{equation}
  for any $j$  distinct from $i, k$. 
  
  At this point of the argument, we use the following easy general fact:
  \begin{lemma} \label{lemma-easy} Let $(E, J)$ be a complex vector space of any dimension, $V$ a {\em real} subspace of $E$ and $v$ an element of $E$ such that $Jv$ belongs to $\mathbb{R} \, v + V$. Then, $v$ belongs to $V + JV$. 
  \end{lemma}
\begin{proof}By hypothesis, $J v = a \, v + w$, for some real number $a$ and some element $w$ of $V$. If $a = 0$, then $v = - J w$ belongs to $JV$. If $a \neq 0$, then $v = a ^{-1} Jv - a ^{-1} w$, hence $J v = - a ^{-1} v - a ^{-1} J w $. Since we also have $J v = a \, v + w$, it follows that $(a + a ^{-1}) v = - w - a ^{-1} Jw$. As $(a + a ^{-1}) \neq 0$, we infer that $v = - (a + a ^{-1}) ^{-1} \, (w + a ^{-1} w)$ belongs to $V + JV$. This concludes the proof of the lemma.
\end{proof}

  By using Lemma \ref{lemma-easy} for $V:={\rm span} (e_i, e _k, J_1 e_i, J_2 e_i, J _3 e _i)$, we readily infer from (\ref{barJ-1}) that
  $e_j$ belongs to $V + {J} V$, for any $j$ distinct from $i, k$, so actually for any $j$, as $e_i$ and $e _k$ already belong to $V$. We would then  eventually get:\begin{equation} V + {J} V = T \mathbb{HP} ^q. \end{equation}
  On the other hand, $V + J V$ is  generated by $e _i, e _k, J _1 e_i, J _2 e _i, J _3 e _i, {J}e _k$, hence is of dimension at most equal to $6$, whereas the dimension of $T \mathbb{HP} ^q$, is equal to $4 q \geq 8$. This contradiction completes the proof of Proposition \ref{prop-HPk}.
  \end{proof} 

\section{Open questions}

While writing these notes, we have encountered several natural questions about metrics admitting orthogonal coordinates whose answers are unknown to us. We list some of them below:

-- Is there any topological obstruction for the existence of metrics with orthogonal coordinates, or does every smooth manifold carry such metrics? 

-- A Riemannian product of Riemannian manifolds with orthogonal coordinates also has orthogonal coordinates. Conversely, if a Riemannian product has orthogonal coordinates, does this hold for the two factors?

-- For a given Riemannian metric, can one find obstructions (in terms of the curvature tensor) to the existence of orthogonal coordinates, other than those given by \eqref{R-2-ijkl}? Note that the Fubini-Study metric on  $\mathbb{C P} ^2$ carries local orthonormal frames satisfying \eqref{R-2-ijkl}, but no orthogonal coordinates (by Proposition \ref{prop-2}).

-- Is every locally symmetric space carrying orthogonal coordinates locally conformally flat? The results in this paper constitute some evidence in favor of a positive answer to this question.

 \end{document}